\documentclass{article}

 \usepackage[preprint]{neurips_2025}

\usepackage[utf8]{inputenc} %
\usepackage[T1]{fontenc}    %
\usepackage{hyperref}       %
\usepackage{url}            %
\usepackage{booktabs}       %
\usepackage{amsfonts}       %
\usepackage{nicefrac}       %
\usepackage{microtype}      %
\usepackage{xcolor}         %

\usepackage{amsmath,amssymb,amsthm}

\usepackage{subcaption}
\usepackage{enumitem}

\usepackage{todonotes}

\usepackage{color-edits}
\addauthor{gb}{blue}

\newcommand{\mathsep}{,~}
\newcommand{\st}{\,\middle|\,}
\newcommand{\set}[1]{\left\lbrace #1 \right\rbrace}
\newcommand{\card}[1]{\left\lvert{#1}\right\rvert}
\newcommand{\absv}[1]{\card{#1}}
\newcommand{\norm}[2]{\left\lVert{#1}\right\rVert_{#2}}

\newcommand{\setR}{\mathbb R}
\newcommand{\setN}{\mathbb N}

\newcommand{\pb}[1]{\mathbb P\left[#1\right]}
\newcommand{\pdf}[1]{{\bf p}\left[#1\right]}

\DeclareMathOperator*{\argmin}{arg\,min}

\newtheorem{definition}{Definition}
\newtheorem{proposition}{Proposition}
\newtheorem{theorem}{Theorem}
\newtheorem{lemma}{Lemma}

\newtheorem{corollary}{Corollary}

\newtheorem{assumption}{Assumption}

\newcommand{\Loss}{\textsc{Loss}}

\newcommand{\sigmoid}{\textsc{Sigmoid}}

\newcommand{\cA}{\mathcal A}
\newcommand{\cB}{\mathcal B}
\newcommand{\cC}{\mathcal C}
\newcommand{\cD}{\mathcal D}

\newcommand{\cR}{\mathcal R}

\newcommand{\cU}{\mathcal U}

\newcommand{\bA}{{\bf A}}

\newcommand{\bD}{{\bf D}}

\DeclareMathOperator{\proj}{proj}

\title{On Monotonicity in AI Alignment}

\author{%
  Gilles Bareilles\thanks{Equal contribution. Correspondance to gilles.bareilles@fel.cvut.cz, julien.fageot@gmail.com, len@calicarpa.com} \\
  CTU in Prague, Tournesol \\
  \And
  Julien Fageot${}^{*}$ \\
  Tournesol \\
  \And
  Lê-Nguyên Hoang${}^{*}$ \\
  Calicarpa, Tournesol \\
  \AND
  Peva Blanchard\thanks{Equal contribution.} \\
  Kleis Technology\\
  \And
  Wassim Bouaziz${}^{\dagger}$ \\
  École Polytechnique \\
  \And
  Sébastien Rouault${}^{\dagger}$ \\
  Calicarpa  \\
  \AND
  El-Mahdi El-Mhamdi \\
  École Polytechnique \\
}

\begin{document}
\maketitle

\begin{abstract} 
Comparison-based preference learning has become central 
to the alignment of AI models with human preferences. 
However, these methods may behave counterintuitively.
After empirically observing that, 
when accounting for a preference for response $y$ over $z$, 
the model may actually decrease the probability (and reward) of generating $y$
(an observation also made by others),
this paper investigates the root causes of (non) monotonicity,
for a general comparison-based preference learning framework 
that subsumes Direct Preference Optimization (DPO), 
Generalized Preference Optimization (GPO) 
and Generalized Bradley-Terry (GBT). 
Under mild assumptions, 
we prove that such methods still satisfy what we call \emph{local pairwise monotonicity}.
We also provide a bouquet of formalizations of monotonicity,
and identify sufficient conditions for their guarantee,
thereby providing a toolbox to evaluate how prone learning models are to monotonicity violations. 
These results clarify the limitations of current methods and 
provide guidance for developing more trustworthy preference learning algorithms. 
\end{abstract}

\section{Introduction}
\label{sec:introduction}

Large AI models and large language models (LLMs) in particular now power an ever‑growing range of user‑facing applications, from conversational assistants to code‑completion systems, and their societal impact expands with every deployment. Ensuring that these models behave in accordance with human preferences has therefore become a defining challenge. Comparison‑based preference learning, in which annotators rank or choose among candidate outputs and the model is fine‑tuned to reproduce those choices, has emerged as the workhorse paradigm for alignment. Although simple to describe and remarkably effective in practice, this paradigm conceals subtle theoretical pitfalls that undermine our ability to reason about, and ultimately trust, the models it produces.

The most widely used framework for comparison-based preference learning is Reinforcement Learning from Human Feedback (RLHF)\cite{DBLP:conf/nips/ChristianoLBMLA17,DBLP:conf/nips/StiennonO0ZLVRA20}, which in practice often reduces to Direct Preference Optimization (DPO)\cite{dpo} or its recent generalizations~\cite{DBLP:conf/icml/TangGZCMRRVPP24,DBLP:conf/aistats/AzarGPMRVC24,fageotGeneralizedBradleyTerryModels2024}. The core intuition behind these methods is straightforward: if a human prefers response $y$ over response $z$, the fine-tuned model should boost the likelihood of $y$ and suppress that of $z$. However,
perhaps surprisingly, recent empirical work has shown that this intuition can fail in practice. In some cases, fine-tuning on a preference pair where $y$ beats $z$ actually reduces the model’s probability or logit score for $y$~\cite{DBLP:journals/corr/abs-2402-13228,DBLP:journals/corr/abs-2410-08847}. Such counterintuitive properties raise serious concerns: they erode trust in the training procedure, complicate the design of data-collection protocols, and may even incentivize annotators to misreport their true preferences, in high-stakes applications.
These phenomena call for a fundamental question:

\begin{center}
    \emph{What monotonicity guarantees do comparison-based preference learning algorithms provide?}
\end{center}

In this paper, we provide the first systematic study of monotonicity for a broad class of comparison-based preference learning methods, 
which includes DPO, Generalized Preference Optimization (GPO), and Generalized Bradley-Terry (GBT). 
More specifically, our contributions are:

\begin{itemize}
    \item We document an empirical setting 
    where individual gradient-descent monotonicity fails.
    \item We formalize a rich variety of flavors of \emph{monotonicity},
    structured around various considerations 
    (pairwise/individual, local/global, score/probability, minimum/gradient-descent).
    \item We prove that,
    a general comparison-based preference learning framework,
    which includes DPO, GPO and GBT,
    guarantees \emph{local pairwise monotonicity}.
    \item We identify sufficient conditions for, global pairwise,
    local individual-score,
    local individual-probability
    gradient-descent pairwise,
    gradient-descent individual-score
    and gradient-descent individual-probability monotonocity.
\end{itemize}

The rest of the paper is organized as follows.
Section~\ref{sec:context} reviews related work.
Section~\ref{sec:experiments} motivates our research, 
by exhibiting an empirical setting where monotonicity fails.
Section~\ref{sec:model} introduces a general comparison-based preference learning framework that generalizes the most leading solutions.
Section~\ref{sec:monotonicity} presents our main result, on \emph{local pairwise monotonicity}.
Section~\ref{sec:secondary_results} discusses other forms of monotonicity.
Section~\ref{sec:conclusion} concludes.

\begin{figure}
    \centering
    \includegraphics[width=\linewidth]{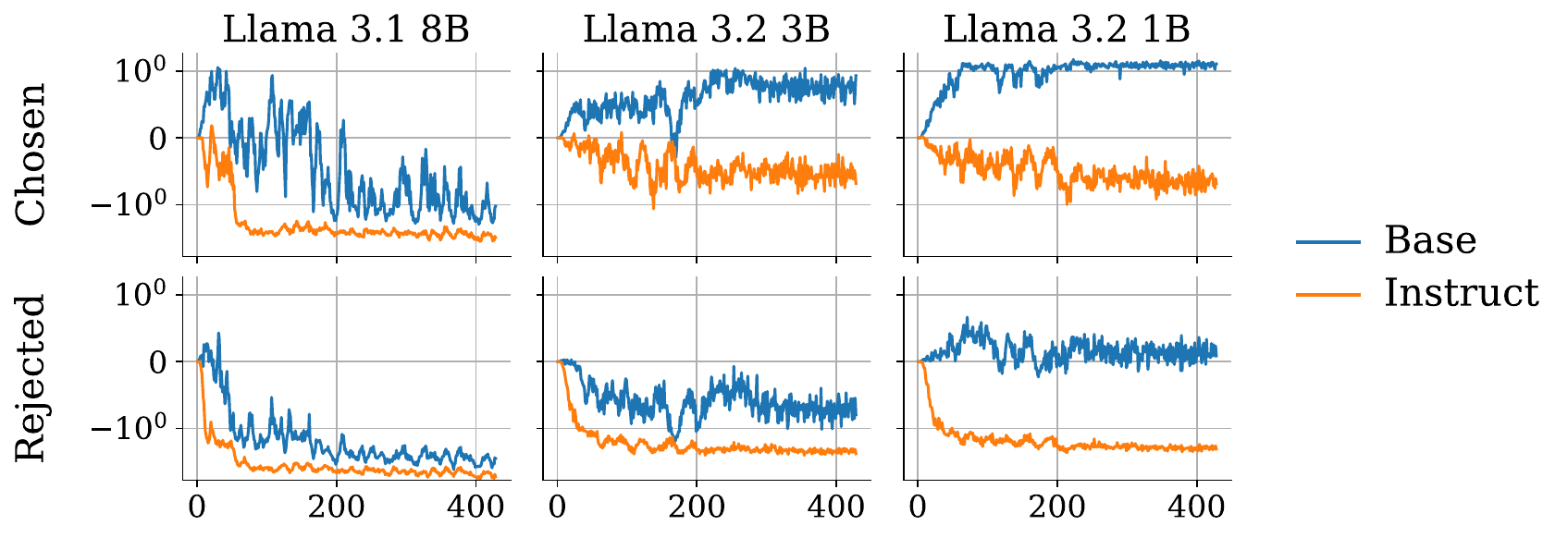}
    \caption{For each model Llama, at each step, 
    we report the difference of scores, before and after the gradient step,
    of, respectively, the chosen and the rejected response.
    One could expect the chosen-response curves to be above zero, 
    and others to be below zero. This is not the case.}
    \label{fig:experiment-chosen-rejected}
\end{figure}

\section{Context and Motivations}
\label{sec:context}

\paragraph{The Bradley-Terry model and its generalizations.}
Comparison-based preference learning builds upon a large literature,
which started with the seminal works of Thurstone~\cite{thurstone27}, 
Zermelo~\cite{zermelo1929berechnung},
and then Bradley and Terry~\cite{bradley1952rank}.
Their solution relies on a probabilistic model of 
how some ground-truth preference gets distorted into reported comparative judgments,
thereby enabling preference learning from inconsistent data.
Their model was later generalized by~\cite{luce1959individual} and \cite{plackett1975analysis}
to account for the selection of one preferred alternative out of many,
by~\cite{DBLP:journals/corr/abs-1911-11658} and \cite{fageotGeneralizedBradleyTerryModels2024}
to enable quantified comparative judgments,
and by~\cite{DBLP:journals/nca/MenkeM08},
~\cite{DBLP:conf/ijcai/GuoTKOCCEDI18},
~\cite{DBLP:conf/aaai/NoothigattuGADR18} and \cite{DBLP:journals/pacmhci/LeeKKKYCSNLPP19}
to learn linear models of preferences, 
and thus generalize preference learning beyond the specific compared items.

\paragraph{Nonlinear models with a Bradley-Terry loss.}
\cite{csiszar2012algorithms} and \cite{zhao2016deep} are some of the earliest nonlinear models
whose loss functions are constructed based on comparative judgments
and on the Bradley-Terry loss.
More recently, 
with the rise of language models~\cite{vaswani2017attention,brown2020language} 
and of the alignment problem~\cite{hadfield2019incomplete,DBLP:conf/aaai/Hoang19},
the Bradley-Terry loss was proposed to fine-tune language models 
to reported comparative human judgments,
e.g. through the convoluted \emph{Reinforcement Learning with Human Feedback} (RLHF)~
\cite{DBLP:conf/nips/ChristianoLBMLA17,DBLP:conf/nips/StiennonO0ZLVRA20}.
This approach was later shown to be reducible to 
\emph{Direct Preference Optimization} (DPO)~\cite{dpo},
where model fine-tuning boils down to minimizing a Bradley-Terry-derived loss function
of the language model parameters.
Lately, alternative loss functions were proposed,
which typically replace the Bradley-Terry loss with an alternative term~
\cite{DBLP:conf/icml/TangGZCMRRVPP24,DBLP:conf/aistats/AzarGPMRVC24}.
The global preference-learning framework has also been used for other use cases, 
like image captioning~\cite{DBLP:journals/corr/abs-2403-06735}
and policy tuning~\cite{DBLP:journals/corr/abs-2310-13639},
as well as
image~\cite{DBLP:conf/cvpr/LiangHLLKCSPYYK24,DBLP:conf/nips/LiuZL0ZH24}
sound~\cite{DBLP:conf/nips/ZhangLLZWZQ24}
and video generation~\cite{DBLP:conf/nips/DaiCWYCJ024}.

\paragraph{Monotonicity.}
While RLHF and DPO have by now been widely used to align language models,
little is known about their actual mathematical guarantees.
For instance, recently, 
\cite{DBLP:conf/nips/ChenMZ0ZRC24} pointed out that order often failed to be recovered
by preference learning algorithm.
More strikingly, \cite{DBLP:journals/corr/abs-2402-13228,DBLP:journals/corr/abs-2410-08847} made observations akin to ours,
as they also witness a decrease of the probability of the preferred alternative,
after including the comparison that says that it is preferred in gradient descent.
In fact, there is a growing literature on fixes to the DPO loss~\cite{DBLP:conf/nips/PangYHCSW24,DBLP:conf/nips/LiuLZ0GYB024}.
Our paper's approach most resembles that of~\cite{fageotGeneralizedBradleyTerryModels2024},
as we study the monotonicity of the loss minimum,
upon the addition or modification of a reported comparative judgment.
We believe this to yield a complementary, and perhaps more fundamental, insight 
than the study of gradient descent.

\paragraph{Experiments.}
\label{sec:experiments}
We replicated the findings previous,
by experimenting with 6 Llama models (3.1 8B, 3.2 3B, 3.2 1B, both \emph{base} and \emph{instruct} variants)~\cite{llama3modelcard} and UltraFeedback~\cite{cui2024ultrafeedbackboostinglanguagemodels}.
We used torchtune~\cite{torchtune} with a modified ``\texttt{full\_dpo\_distributed}'' \emph{recipe} (provided in the Supplementary Material).
Our experiments ran on a compute node of 8 H100, for less than 100 GPU-hour.
Figure~\ref{fig:experiment-chosen-rejected} shows no guarantee of monotonicity.
Namely, the scores of the chosen response may increase or decrease,
while the score of the rejected response may also increase or decrease.
It is noteworthy that the base model tends to respect monotonicity more than the instruct model does,
though this observation is far from robust.
Such puzzling results call for a theory of monotonicity.

\section{Model}
\label{sec:model}

In this section, we introduce a very general comparison-based preference learning framework.
We show that it includes the most celebrated instances,
including Bradley-Terry (BT), Generalized Bradley-Terry (GBT), 
Bradley-Terry-based linear models, 
Direct Preference Optimization (DPO) and General Preference Optimization (GPO).

Consider a set $\cA$ of alternatives to be scored.
We assume that their scoring is dependent on a background $\cB$.
Typically, in the context of language model alignment,
$\cB$ would be the set of prompts and $\cA$ would be the set of responses to the prompt.
Denote $s : \cA \times \cB \times \setR^D \to \setR$
the parameterized scoring function to be learned,
where $s_{y|x} (\theta) \in \setR$ is the score 
assigned to alternative $y \in \cA$ to background $b \in \cB$
for a parameter vector $\theta \in \setR^D$.

The parameter vector $\theta$ is typically learned by fitting
a comparison-based preference multiset 
$\bD \triangleq \left( \cB \times \cA \times \cA \times \cC \right)^*$
composed of a finite number of conditional pairwise response comparisons $(x, y, z, c)$,
where $x \in \cB$ is the background (e.g. prompt),
$y, z \in \cA$ are proposed alternatives (e.g. responses) to $x$,
and $c \in \cC \subset \setR$ says 
whether $y$ was preferred over $z$ ($c > 0$), or $z$ was preferred over $y$ ($c < 0$).
Typically, assuming binary comparisons,
we would have $\cC \triangleq \set{-1, +1}$, 
with $c = 1$ if $y$ was preferred to $z$,
and $c = -1$ otherwise.

To fit $\theta$ to $\bD$, we assume that a loss is minimized.
Denoting $s_{yz|x} (\theta) \triangleq s_{y|x} (\theta) - s_{z|x} (\theta)$
the score difference between responses $y$ and $z$ on prompt $x$,
we consider the following general loss form:
\begin{align}
    \Loss (\theta | \bD)
    &= \cR(\theta) + \sum_{(x, y, z, c) \in \bD} \ell(s_{yz | x} (\theta), c),
\end{align}
where $\cR : \setR^D \to \setR$ is a (potentially nil) regularization
and $\ell: \setR \times \cC \to \setR$ is the loss per data point.

In the sequel, we show that our setting generalizes most state-of-the-art solutions
for comparison-based preference learning,
which are obtained by instantiating different scoring functions $s$
and different per-data losses $\ell$.
Note however that some models escape our formalism, 
e.g.~\cite{DBLP:journals/corr/abs-2402-13228,DBLP:conf/nips/XiaoYZLH24,DBLP:conf/nips/0001X024} 
whose losses also depend on $s_{y|x} (\theta)$ or $\pi_\theta(y|x)$, 
and not just on the score difference.

\subsection{Variants of the scoring function $s$}

\paragraph{One-hot encoding.}
The simplest instantiation of $s$ simply corresponds 
to a parameter vector $\theta \in \setR^{\cB \times \cA}$,
and $s_{y|x} (\theta) \triangleq \theta_{xy}$.
This corresponds to one-hot encoding,
as it can be rewritten $s_{y|x} (\theta) \triangleq \theta^T (e_x \otimes e_y)$,
where $e_x$ and $e_y$ are the vector of the canonical bases of $\setR^{\cB}$ and $\setR^{\cA}$.
Unfortunately, however, one-hot encoding fails to perform generalization.
Namely, the knowledge that $y$ has high score under $x$
does not affect the scoring of $y'$ under $x$,
even if $y'$ is known to be very similar to $y$.
Additionally, in applications like language model fine-tuning 
where $\cB$ or $\cA$ are combinatorially large,
one-hot encoding requires an exponential number of parameters, which is impractical.

\paragraph{Linear model.}
A more common scoring function of $s$ in machine learning involves a linear model.
To do so, we first consider a fixed embedding map 
$f: \cB \times \cA \to \setR^D$.
The score is then a linear function of the embedding, i.e.
$s_{y|x} (\theta) = \theta^T f(x,y)$.
This is, to a certain extent, what is performed in the context of
Reinforcement Learning with Human Feedback (RLHF),
where the score (also known as reward) is constructed as a linear function of an embedding.
However, note that this is only one step of RLHF, 
which also involves policy optimization given a scoring function.

\paragraph{Language models.}
For language models,
we have $\cA = \cB = \bA^* \triangleq \bigcup_{n \in \setN} \bA^n$,
i.e. both the alternatives and the background are finite sequences of characters
of a finite alphabet $\bA$.
The scoring function then assigns a score $s_{y|x} (\theta) \in \setR$
to any response (alternative) $y \in \cA$
under a prompt (background) $x \in \cB$.
It typically corresponds to the last layer of the language model,
before a softmax operator is applied to derive a probability distribution over $\cA$,
i.e. it is common to set
\begin{align}
    \pi_\theta (y|x)
    \triangleq \frac{\exp(s_{y|x}(\theta))}{
        \sum_{z \in \bA^*} \exp(s_{z|x}(\theta))
    },
\end{align}
where $\pi_\theta (y|x)$ is the probability of response $y$ under prompt $x$.
If so, the scores $s_{y|x} (\theta)$ are known as the \emph{logits} of the generative model.

\paragraph{Direct Preference Optimization (DPO).}
In Direct Preference Optimization (DPO),
which is an equivalent more direct reformulation of RLHF,
a reference model $\pi_{ref} : \bA^* \to \Delta(\bA^*)$ is used
to bound the variations of the scores.
The score $s_{y|x} (\theta)$ 
to response $y$ conditionally to prompt $x$ assuming model $\theta$
is then given by%
\begin{align}
  s_{y|x} (\theta)
  = \beta \log \frac{\pi_\theta(y | x)}{\pi_{ref} (y | x)} + \beta \log Z_x(\theta),
\end{align}
where $Z_{x}(\theta) = \sum_{y} \pi_{ref}(y | x) \exp(\beta^{-1}s_{y|x}(\theta)) $ 
is the partition function of $\pi_{\theta}(\cdot|x)$, 
and $\beta \in \setR_{\geq 0}$ is a positive scalar hyperparameter.
Note that $s_{y|x} (\theta)$ is here often known as the \emph{reward}.

\paragraph{}
In all these cases, $s_{y|x}$ is often assumed to be differentiable, if not smooth\footnote{
Modern language models typically consider the smooth Sigmoid Linear Unit (SiLU) function as an activation function,
instead of, say, ReLU.
}.
In the sequel, we will assume that it is continuously differentiable.

\begin{assumption}
\label{ass:s_c1}
    For all $x \in \cB$ and $y \in \cA$, the function $s_{y|x}: \setR^D \to \setR$ is continuously differentiable.
\end{assumption}

\subsection{Variants of the loss function $\ell$}

\paragraph{Bradley-Terry (BT).}
In DPO, and many other comparison-based preference learning models,
the probability that $y$ is preferred to $z$ is then given 
by the classical Bradley-Terry model~\cite{bradley1952rank}, i.e.
\begin{align}
  \pb{c = 1 | x, y, z, \theta}
  \triangleq \sigmoid \left( s_{yz|x} (\theta) \right), 
  \quad
  \pb{c = -1 | x, y, z, \theta}
  \triangleq \sigmoid \left( -s_{yz|x} (\theta) \right),
\end{align}
where $\sigmoid(t) \triangleq 1/(1+e^{-t})$ is the sigmoid function
and $s_{yz|x} (\theta) \triangleq s_{y|x} (\theta) - s_{z|x} (\theta)$
is the score difference between responses $y$ and $z$.
Assuming that the prompts and answers $x$, $y$ and $z$ are independent from $\theta$,
the negative log-likelihood then defines a loss $\ell$, 
up to a constant independent from $\theta$,
which is given by
$
  \ell (s_{yz|x}(\theta), c)
  \triangleq - \log \pb{c | x, y, z, \theta}
  = - \log \sigmoid \left( c s_{yz|x} (\theta) \right)$.
Or to put it more straightforwardly, we have 
\begin{align}
    \ell(s, c) = - \log \sigmoid (cs).
\end{align}

    Note that minimizing the above loss for the simplest dataset $\bD = (x, y, z, 1)$, 
    amounts to maximizing $\sigmoid(s)$.
    Since the sigmoid function is increasing,
    this corresponds to high values of $s$.
    In the DPO setting,
    one recovers that this favors increasing $\pi_{\theta} (y | x)$ and decreasing $\pi_{\theta} (z | x)$.

\paragraph{Generalized Bradley-Terry.}
The DPO and Bradley-Terry models handle ``binary'' comparisons, namely $c=1$ or $c=-1$.
In many situations though, one can say whether $y$ is preferable to $z$, but also by \emph{how much}.
\cite{fageotGeneralizedBradleyTerryModels2024} proposed a family of Generalized Bradley-Terry (GBT) models,
that allow including quantified comparisons $c \in \cC$,
where $\cC \subset \setR$ is symmetric with respect to $0$; 
typically, $\cC = [ -1, 1 ]$ or $\cC = \setR$.
Given a score difference $s_{yz|x}$,
a GBT model induces the following distribution of comparisons $c$:
\begin{align}
  \pdf{c | x, y, z, \theta}
  \triangleq \frac{
    f(c) \exp \left(c s_{yz|x} (\theta) \right)
  }{
    \int_\cC f(\gamma) \exp \left(\gamma s_{yz|x} (\theta) \right) d\gamma
  },
\end{align}
where $f$ is a ``root law'' distribution over $\cC$ that characterizes the GBT model.
Note that the classical Bradley-Terry model is recovered by setting $\cC = \set{-1, +1}$ and $f = (\delta_{-1} + \delta_{1}) / 2$, where $\delta_{p}$ denotes the Dirac distribution at $p$.
From this we can derive the loss 
$\ell (s_{yz|x}(\theta), c) \triangleq - \log \pdf{c | x, y, z, \theta} + cst$ 
as the negative log-likelihood of the data (up to a constant), 
we obtain
\begin{align}
  \ell(s, c) = \Phi_f \left( s \right) - c s,
\end{align}
where $\Phi_f(s) = \log \int_\cC e^{s\gamma} f(\gamma) d\gamma$
is the cumulant-generating function of the root law $f$.

\paragraph{Uniform-GBT.}
For $\cC = [-1, 1]$ and $f^{\text{unif}} = 1_{[1, 1]} / 2$,
the loss is given by
\begin{align}
  \ell (s, c)
  = \log \frac{ \sinh \left( s \right) }{s} - c s.
\end{align}

\paragraph{Gaussian-GBT.}
Another interesting case is $\cC = \setR$ with $f(c) = \exp(-c^2/2)$,
which corresponds to a normally distributed root law,
which then yields
\begin{align}
  \ell (s, c)
  = \frac{1}{2} s^2 - c s
  = \frac{1}{2} (s - c)^2 - \frac{1}{2}c^2.
\end{align}
Up to a multiplicative rescaling of the scores, 
this corresponds to the variant of DPO introduced by~\cite{gaussian_dpo},
where $c$ is obtained through a willingness-to-pay mechanism.
We refer to \cite{fageotGeneralizedBradleyTerryModels2024} 
for a table of values of $\Phi_f$ for different root laws $f$.

\paragraph{GPO losses.}
Note that our formulation generalizes General Preference Optimization (GPO)~\cite{DBLP:conf/icml/TangGZCMRRVPP24},
which propose numerous other expressions for the loss $\ell$.
As they only consider binary comparisons, 
they write their function $\ell(s, 1) = \ell_0(s)$, with $\ell(s, -1) = \ell_0(-s)$.
Various forms of $\ell_0$ are considered, 
including $\ell_0 = - \log \sigmoid$ (DPO~\cite{dpo}),
$\ell_0(s) = \max(0, 1 - s)$ (SLiC~\cite{zhao2023slic}),
and $\ell_0(s) = (1-s)^2$ (IPO~\cite{DBLP:conf/aistats/AzarGPMRVC24}).
\cite{DBLP:conf/nips/0001HFCFSL24} automatically searched and found more examples.

\section{Pairwise Monotonicity}
\label{sec:monotonicity}

In this section, we formalize \emph{pairwise monotonicity},
and we essentially prove that all models that are instances of our general framework
guarantee \emph{local} \emph{pairwise monotonicity}.

\subsection{Defining monotonicity}

Intuitively, monotonicity holds if,
whenever a preference for response $y$ over $z$ is reported,
the model trained with this preference will improve the scoring of $y$ over $z$.
However, precisely formulating this intuition raises a few issues.

First, different statistics of the language models may be tracked to evaluate monotonocity.
Some papers~\cite{DBLP:journals/corr/abs-2402-13228,DBLP:journals/corr/abs-2410-08847} previously looked at the probability $\pi_{\theta}(y|x)$ 
of generating the preferred response given $x$.
This may be called \emph{individual-probability monotonicity}.
One could also be interested to look at the individual score variations: 
increase of $s_{y|x}(\theta)$ and decrease of $s_{z|x}(\theta)$.
We may call this criterion \emph{individual-score monotonicity}.
We will discuss these notions later on, and will show that they do not hold in general.
In this section, we rather focus on the difference of scores 
$s_{yz|x}(\theta) = s_{y|x}(\theta) - s_{z|x}(\theta)$
between the preferred and the less preferred responses.
We call this \emph{pairwise monotonicity}.
Assuming that scores are the logits of the generation probabilities, 
pairwise monotonicity then implies a monotonicity of probability ratios, as
\begin{align}
    s_{yz|x} (\theta^{(2)}) \geq s_{yz|x} (\theta^{(1)})
    ~\Longleftrightarrow~
    \frac{\pi_{\theta^{(2)}}(y|x)}{\pi_{\theta^{(2)}}(z|x)} \geq \frac{\pi_{\theta^{(1)}}(y|x)}{\pi_{\theta^{(1)}}(z|x)}.
\end{align}

Second, monotonicity may be measured either 
with respect to an intensification of a comparison,
or to the addition of a unequivocal comparison. 
The former will be the subject of Section~\ref{sec:comparison_intensification},
while the latter will be that of Section~\ref{sec:unequivocal_comparison}.

Third, in the general case, 
it is unclear what it means for a language model to learn 
from the addition of a data in its dataset,
especially if the loss function has multiple minima.
To mitigate this concern,
we only consider infinitesimal deviations from a strict minimum,
with a positive definite Hessian loss.
In particular, we only consider the addition of a comparison
with an infinitesimal weight. 
This yields what we call \emph{local} monotonicity.
This scenario is arguably not far from practice 
given the number of data points used for training these models.

\subsection{Pairwise monotonocity when adding a unequivocal comparison}
\label{sec:unequivocal_comparison}

In this section, we assume that $\cC$ is bounded, hence has a maximum. 
This typically includes the settings where $\cC$ is finite like Bradley-Terry, DPO and GPO,
as well as GBT with a uniform root law on an interval or on a finite set,
among many others possibilities.
We then consider adding a small-weight data to $\bD$, 
by defining
$\bD' \triangleq \bD \cup \varepsilon \set{(x, y, z, \max \cC)}$,
where $\bD'$ now has $N+1$ data, 
the last of which being $(x, y, z, \max \cC)$ with a weight $\varepsilon$ when it appears in $\Loss$.
Formally,
\begin{align}
    \Loss(\theta | \bD') \triangleq \Loss(\theta | \bD) + \varepsilon \ell(s_{yz|x}(\theta), \max \cC).
\end{align}

\begin{definition}
  A preference learning model is \emph{locally pairwise monotone} 
  at dataset $\bD$ and parameters $\theta^* \in \setR^D$ 
  for the addition of a unequivocal comparison $(x, y, z, \max \cC)$, if
  it is based on minimizing a loss function $\Loss$ and if
  there exists a neighborhood $\cU$ of $\theta^*$ and $\varepsilon_0 > 0$ such that,
  for all $x, y, z \in \bA^*$ and
  for all $0 \leq \varepsilon \leq \varepsilon_0$,
  \begin{equation}
    \forall \theta^\varepsilon \in \argmin_{\theta \in \cU} 
       \Loss ( \theta | \bD \cup \varepsilon \set{(x, y, z, \max \cC)}) 
       \mathsep
    s_{yz|x} (\theta^{\varepsilon}) \geq s_{yz|x} (\theta^*) 
  \end{equation}
\end{definition}

Intuitively, for local pairwise monotonicity to hold,
a maximal comparison must push for larger score differences between $y$ and $z$.
Formally, this amounts to the following.

\begin{assumption}
    \label{ass:max_comparison}
    The loss $\ell: \setR \times \cC \to \setR$ 
    is twice continuously differentiable in its first variable,
    and so is the regularization $\setR$.
    Moreover, the set $\cC$ has a maximum and 
    $\partial_s \ell(s, \max \cC) < 0$
    for all $s \in \setR$.
\end{assumption}

Some versions of GPO do not verify Assumption~\ref{ass:max_comparison},
in particular for SLiC (not twice continuously differentiable)
and for IPO (where saying that $y$ is preferred over $z$ pulls the score difference towards 1,
even if the score difference would otherwise be larger than 1).
However, the assumption holds for the classical Bradley-Terry model,
and more generally, for all generalized Bradley-Terry models with a maximal comparison.

\begin{proposition}
\label{prop:gbt_max_comparison}
    Assume that $\cC$ has a maximum
    and that $\ell$ is derived from the Generalized Bradley-Terry model:
    there exists a root law $f: \cC \to \setR_{\geq 0}$ such that
    $\ell(s, c) = \Phi_f(s) - cs$.
    Then $\partial_s \ell(s, \max \cC) < 0$ for all $s \in \setR$.
\end{proposition}

\begin{proof}
    The loss of the GBT model with root law $f$ is $\ell(s,c) = \Phi_f(s) - cs$, 
    hence 
    $\partial_s \ell(s, \max \cC) = \Phi'_f(s) - \max \cC$. 
    The derivative of the cumulant generative function is known to be a strictly increasing odd bijection from $\mathbb{R}$ to $(\min \cC, \max \cC)$ \cite[Theorem 1]{fageotGeneralizedBradleyTerryModels2024}.
    Hence, $\Phi'_f(s) - \max \cC < 0$.
\end{proof}

\begin{theorem}
  \label{th:monotonicity_unequivocal}
  Consider a preference learning model 
  that meets Assumption~\ref{ass:s_c1} and Assumption~\ref{ass:max_comparison}, and a dataset $\bD$.
  Let $\theta^* \in \setR^D$ and $(x, y, z) \in \cB \times \cA \times \cA$.
  If
  $\nabla \Loss (\theta^* | \bD) = 0$,
  $\nabla^2 \Loss (\theta^* | \bD) \succ 0$
  and $\nabla s_{zy|x} (\theta^*) \neq 0$,
  then $\Loss$ is locally pairwise monotone at $\bD$ and $\theta^*$
  for the addition of the unequivocal comparison $(x, y, z, \max \cC)$.
\end{theorem}

\begin{proof}[Proof sketch]
  The proof leverages the implicit function theorem,
  applied to the equality $\nabla \Loss (\theta^\varepsilon | \bD^\varepsilon) = 0$,
  which implies
  \begin{align*}
      s_{yz|x} (\theta^\varepsilon) - s_{yz|x} (\theta^*)
      = - \varepsilon \partial_s \ell(s_{yz|x} (\theta^*), \max \cC)
        \nabla_\theta s_{yz|x}^T \left[ \nabla^2 \Loss(\theta^*|\bD) \right]^{-1} \nabla_\theta s_{yz|x} 
        + o(\varepsilon).
  \end{align*}
  A sign analysis then allows to conclude.
  The full proof is given in Appendix~\ref{app:unequivocal}.
\end{proof}

While Theorem~\ref{th:monotonicity_unequivocal} applies to many different comparison-based preference learning schemes,
for the sake of exposition,
we state its implication for the most popular setting, namely DPO.

\begin{corollary}
\label{cor:dpo_monotonocity}
    Consider DPO 
    with a local minimum $\theta^*$ at which the Hessian matrix of the loss is positive definite.
    Then DPO is locally pairwise monotone at $\theta^*$ 
    with respect to the addition of any unequivocal comparison $(x, y, z, \min \cC)$
    for which $\nabla s_{zy|x} (\theta^*) \neq 0$.
\end{corollary}

\begin{proof}
    As DPO uses a Bradley-Terry loss, which is a particular instance of GBT,
    it verifies Assumption~\ref{ass:max_comparison} (Proposition~\ref{prop:gbt_max_comparison}).
    Theorem~\ref{th:monotonicity_unequivocal} then applies.
\end{proof}

\subsection{Pairwise monotonocity with respect to comparison intensification}
\label{sec:comparison_intensification}

We now consider monotonicity under comparison intensification.
Namely, we fix a triple $(x, y, z) \in \cB \times \cA \times \cA$.
For any given comparison $(x', y', z', c') \in \cB \times \cA \times \cA \times \cC$,
we define the $\varepsilon$-intensification of the comparison $c$ in favor of $y$ against $z$
under $x$ by
\begin{align}
    \textsc{push}_\varepsilon^{x, y, z} \left(c' \st x', y', z' \right)
    \triangleq
  \begin{cases}
    \proj_{\cC}(c' - \varepsilon)  & \text{if } (x', y', z') = (x, z, y), \\
    \proj_{\cC}(c' + \varepsilon) & \text{if } (x', y', z') = (x, y, z), \\
    c' & \text{otherwise},
  \end{cases}
\end{align}
where $\proj_{\cC} (t) \triangleq \argmin_{c \in \cC} \absv{t - c}$
is the projection on $\cC$.
Informally, any comparison between $y$ and $z$ on prompt $x$
is given a slight preference move towards $y$,
while other comparisons are left unchanged.
The $\varepsilon$-intensified dataset is then 
\begin{align}
  \bD + \Delta^{\varepsilon}_{y z|x}
  &\triangleq \set{
    \left(x, y, z, \textsc{push}_\varepsilon^{x, y, z} \left(c' \st x', y', z' \right) \right)
    \st (x', y', z', c') \in \bD
  }.
\end{align}

\begin{definition}
A loss $\Loss$ with dataset $\bD$ is locally pairwise monotone at a local minimum $\theta^*$ for comparison intensification, if
there exists a neighborhood $\cU$ of $\theta^*$ and $\varepsilon_0 > 0$ such that,
for all $x, y, z \in \bA^*$,
for all $0 < \varepsilon \leq \varepsilon_0$,
we have 
\begin{align}
    \forall \theta^\varepsilon \in \argmin_{\theta \in \cU} \Loss ( \theta | \bD + \Delta^\varepsilon_{yz|x}) \mathsep
    s_{yz|x} (\theta^{\varepsilon}) \geq s_{yz|x} (\theta^*)
\end{align}
\end{definition}

The following assumption will help us characterize 
a family of locally pairwise-monotone preference learning models.

\begin{assumption}
    The set $\cC$ is an interval of $\setR$.
    Moreover, the loss $\ell : \setR \times \cC \to \setR$ 
    and the regularization $\cR: \setR^D \to \setR$ are twice continuously differentiable,
    and $\partial_c \partial_s \ell(s, c) < 0$ 
    for all score differences $s \in \setR$ and all comparisons $c \in \cC$.
    \label{ass:differentiable}
\end{assumption}

The latter assumption implies that $\partial_s \ell (s, c)$ is a decreasing function of $c$.
Among all the examples we introduced in Section~\ref{sec:model},
the only cases where $\cC$ is an interval are the GBT losses.
It turns out that all these losses verify Assumption~\ref{ass:differentiable}.

\begin{proposition}
\label{prop:differentiable}
    Any GBT loss whose root law has an interval support verifies Assumption~\ref{ass:differentiable}.
    This includes, for instance, Uniform-GBT and Gaussian-GBT.
\end{proposition}

\begin{proof}
    For GBT, $\ell(s,c) = \Phi_f(s) - sc$, hence $\partial_c \partial_s \ell(s, c) = -1 < 0$.
\end{proof}

\begin{theorem}
\label{th:monotonicity_comparison_intensification}
    Under Assumption~\ref{ass:s_c1} and Assumption~\ref{ass:differentiable},
    If $\nabla \Loss (\theta^* | \bD) = 0$, 
    $\nabla^2 \Loss (\theta^* | \bD) \succ 0$
    and $\nabla s_{zy|x} (\theta^*) \neq 0$ for all $(x, y, z, c) \in \bD$,
    then $\Loss$ with dataset $\bD$ is locally pairwise monotone at $\theta^*$,
    for comparison intensification.
\end{theorem}

\begin{proof}[Proof sketch]
  The proof leverages the implicit function theorem
  to provide a first-order approximation of the new scores
  for the dataset $\bD + \Delta_{yz|x}^{\varepsilon}$.
  The full proof is given in Appendix~\ref{app:intensification}.
\end{proof}

\section{Other Forms of Monotonicity}
\label{sec:secondary_results}

We essentially found that infinitesimally favoring $y$ over $z$
implies an increase of the score of $y$ with respect to the score of $z$,
for a wide class of comparison-based preference learning models.
In this section, we analyze other forms of monotonicity.

\subsection{Global Pairwise Monotonicity Under Strong Convexity}

Under appropriate convexity assumptions, we can remove the infinitesimal assumption.

\begin{definition}
    A loss $\Loss$ is globally pairwise monotone if,
    for any dataset $\bD$, 
    any $x, y, z \in \bA^*$, 
    any intensification of comparisons $yz|x$ in $\bD$
    and any number of additions of comparisons $(x, y, z, \max \cC)$
    yielding a modified dataset $\bD'$ that favors more $y$ against $z$ under $x$
    than $\bD$ does,
    \begin{align}
        \forall \theta \in \argmin \Loss(\cdot | \bD) \mathsep
        \forall \theta' \in \argmin \Loss(\cdot | \bD') \mathsep
        s_{yz|x}(\theta') \geq s_{yz|x}(\theta).
    \end{align}
\end{definition}

\begin{assumption}
    \label{ass:strongly_convex}
    The loss $\ell: \setR \times \cC \to \setR$ 
    and the regularization $\cR: \setR^D \to \setR$ are continuously differentiable.
    Moreover, for any $c \in \cC$, 
    and any $(x, y, z) \in \cB \times \cA \times \cA$,
    $\theta \mapsto \ell( s_{yz|x} (\theta), c)$ is convex,
    while $\cR$ is strongly convex on any compact set.
\end{assumption}

Assumption~\ref{ass:strongly_convex} typically holds for $\ell$ convex and $s$ linear in $\theta$.

\begin{theorem}
\label{th:global}
    Suppose Assumptions~\ref{ass:s_c1} and \ref{ass:strongly_convex} hold.
    Then, on one hand, Assumption~\ref{ass:max_comparison} implies global pairwise monotonocity 
    with respect to unequivocal comparisons.
    Meanwhile, on the other hand, Assumption~\ref{ass:differentiable} implies global pairwise monotonocity 
    with respect to comparison intensification.
\end{theorem}

\begin{proof}[Proof sketch]
    Because of strong convexity, the minimum is always unique,
    and can thus be written as a function $\theta^*(\bD)$.
    Now consider a continuous path $f: [0, 1] \to \cD$
    with $f(0) = \bD$, $f(1) = \bD'$ 
    and which continuously adds weights to unequivocal comparisons $yz|x$
    or intensifies the comparisons $yz|x$ in favor of $y$.
    By the implicit function theorem, 
    $\frac{d}{dt} \left[ s_{yz|x} (f(t)) \right] \geq 0$.
    Integrating from $0$ to $1$ yields the conjecture.
    The full proof is given in Appendix~\ref{app:strongly_convex}.
\end{proof}

\subsection{Local Individual-Score Monotonicity}

Instead of score differences,
we could be interested in the preferred alternative score,
as in \cite{fageotGeneralizedBradleyTerryModels2024}.

\begin{definition}
A loss $\Loss$ with dataset $\bD$ is locally individual-score monotone 
at a local minimum $\theta^*$ for comparison intensification, if
there exists a neighborhood $\cU$ of $\theta^*$ and $\varepsilon_0 > 0$ such that,
for all $(x, y, z) \in \cB \times \cA \times \cA$,
for all $0 < \varepsilon \leq \varepsilon_0$,
\begin{align}
    \forall \theta^\varepsilon \in \argmin_{\theta \in \cU} \Loss ( \theta | \bD + \Delta^\varepsilon_{yz|x}) \mathsep
    s_{y|x} (\theta^{\varepsilon}) \geq s_{y|x} (\theta^*)
    ~\text{and}~
    s_{z|x} (\theta^{\varepsilon}) \leq s_{z|x} (\theta^*).
\end{align}
\end{definition}

Similarly to \cite{fageotGeneralizedBradleyTerryModels2024},
we find a sufficient condition based on max-diagonal dominance.

\begin{definition}
    A symmetric matrix $M \in \mathbb{R}^{D\times D}$ is max-diagonally dominant if, for any $i \in [D]$, $M_{ii} \geq \max_{j \neq i} M_{ij}$.
\end{definition}

\begin{theorem}
\label{th:individual_score_monotonicity}
    Under Assumption~\ref{ass:differentiable},
    If $\nabla \Loss (\theta^* | \bD) = 0$, 
    $\nabla^2 \Loss (\theta^* | \bD) \succ 0$
    and $\nabla s_{zy|x} (\theta^*) \neq 0$ for all $(x, y, z, c) \in \bD$.
    We assume moreover that the matrix $ \nabla_\theta s_{|x} (\theta^*)^T \left[ \nabla^2 \Loss (\theta^* | \bD) \right]^{-1} 
            \nabla_\theta s_{|x} (\theta^*) \in \mathbb{R}^{\cA \times \cA}$ is max-diagonally dominant.
    Then $\Loss$ with dataset $\bD$ is locally individual-score monotone at $\theta^*$,
    for comparison intensification.
\end{theorem}

\begin{proof}[Proof sketch]
    The proof, given in Appendix~\ref{app:individual_score_monotonicity}, 
    again leverages the implicit function theorem.
\end{proof}

Unfortunately, max-diagonal dominance is a very demanding property 
especially for large matrices (see~\cite{strong-monotonicity}).
Yet the matrix that is assumed to be max-diagonally dominant 
in Theorem~\ref{th:individual_score_monotonicity}
is of size $\cA \times \cA$.
Yet in the context of language models, 
$\cA$ is the set of possible responses to a prompt,
which is exponentially large in the response length.
This suggests that local individual-score monotonicity is highly unlikely to hold
for any comparison-based language preference learning algorithm.

\subsection{Locally individual-probability monotonicity}

In the context of language models, rather than scores,
it is arguably more meaningful to focus on the monotonicity of probabilities
(or, equivalently, of log-probabilities).
We formalize this for local monotonicity, for any modification of the dataset $\bD$.

\begin{definition}
A loss $\Loss$ with dataset $\bD$ is locally individual-probability monotone 
at a local minimum $\theta^*$ for a modification of $\bD$ into $\bD^\varepsilon$, if
there exists $\varepsilon_0 > 0$ such that,
for all $(x, y, z) \in \cB \times \cA \times \cA$,
for all $0 < \varepsilon \leq \varepsilon_0$,
\begin{align*}
    \forall \theta^\varepsilon \in \argmin_{\theta \in \cU} \Loss (\theta | \bD^\varepsilon) \mathsep 
    \pi_{\theta^{\varepsilon}} (y|x) \geq \pi_{\theta^*} (y|x)
    ~\text{and}~
    \pi_{\theta^{\varepsilon}} (z|x) \leq \pi_{\theta^*} (z|x).
\end{align*}
\end{definition}

We show that this monotonicity is vaguely linked to pairwise monotonicity.
More precisely, it follows from a stronger version of pairwise monotonicity,
which we call \emph{fully pairwise monotonicity}.

\begin{definition}
Fully pairwise monotonicity holds if
\begin{align}
    \forall w \in \cA - \set{y} \mathsep
    s_{yw|x} (\theta^\varepsilon) \ge s_{yw|x} (\theta^*),
\end{align}
i.e. the score difference with any other response increases.
\end{definition}

\begin{proposition}
\label{prop:pairwisetoindividual}
    Assuming probabilities are softmax functions of the scores,
    fully-pairwise monotonicity implies individual-probability monotonicity.
\end{proposition}

\begin{proof}
    The proof follows by simplifying the terms of the fraction $\pi_\theta (y|x)$.
    See Appendix~\ref{app:pairwisetoindividual}.
\end{proof}

Individual-probability and fully pairwise monotonicity are very demanding,
and seem unlikely to hold in practice, even locally,
especially in the context of the language fine-tuning.
Nevertheless, we prove the existence of an algorithm 
that does verify fully-pairwise monotonicity
(and thus individual-probability monotonicity for softmax outputs on the scores).

\begin{proposition}
\label{prop:GBTfullypairwise}
    GBT (with $s(\theta) = \theta$) is globally fully-pairwise monotone
    with respect to both unequivocal comparison addition and comparison intensification.
\end{proposition}

\begin{proof}
    The proof leverages properties of diagonally-dominant matrices.
    See Appendix~\ref{app:GBTfullypairwise}.
\end{proof}

\subsection{Gradient Descent Monotonicity}

So far, our theory focused on local/global monotonicity,
as we believe it to address a more fundamental consideration.
We now circle back to our experiments (Figure~\ref{fig:experiment-chosen-rejected}),
by providing sufficient conditions for \emph{gradient-descent monotonicity} 
for nil regularization $\cR = 0$.

\begin{definition}
    A loss $\Loss$ with $\cR = 0$ is pairwise gradient-descent monotone 
    with respect to the addition of an unequivocal comparison
    at $(x, y, z) \in \cB \times \cA \times \cA$ and $\theta \in \setR^D$,
    if there exists $\varepsilon_0 > 0$ such that 
    for all $0 \leq \varepsilon \leq \varepsilon_0$,
    denoting $\theta^\varepsilon$ the solution after a gradient step 
    with learning rate $\varepsilon$,
    i.e.
    \begin{align}
        \theta^\varepsilon 
        = \theta - \varepsilon \nabla_\theta \left[ \ell(s_{yz|x}(\theta), \max \cC) \right],
    \end{align}
    we have $s_{yz|x} (\theta^\varepsilon) \geq s_{yz|x} (\theta)$.
    Similarly, we define fully-pairwise, individual-score and individual-probability monotonicity,
    by replacing the last condition with, respectively,
    $\forall w \in \cA - \set{y}, s_{yw|x} (\theta^\varepsilon) \geq s_{yw|x} (\theta)$,
    $s_{y|x} (\theta^\varepsilon) \geq s_{y|x} (\theta)$,
    and $\pi_{\theta^\varepsilon} (y|x) \geq \pi_{\theta} (y|x)$,
\end{definition}

\begin{theorem}
\label{th:gradient_descent}
    Make Assumptions~\ref{ass:s_c1} and \ref{ass:max_comparison}.
    Suppose $\cR = 0$.
    Then at any $\theta \in \setR^D$,
    and with respect to the addition of any unequivocal comparison $(x, y, z, \max \cC)$,
    we have the implications:
    \begin{align*}
        &\nabla s_{yz|x} (\theta) \neq 0 &\implies &\text{pairwise gradient-descent monotonicity}, \\
        &\nabla s_{yz|x} (\theta)^T \nabla s_{y|x} (\theta) > 0 &\implies &\text{individual-score gradient-descent monotonicity}, \\
        &\forall w \in \cA - \set{z}, \nabla s_{yw|x} (\theta)^T \nabla s_{yz|x} (\theta) > 0 &\implies &\text{fully-pairwise gradient-descent monotonicity}.
    \end{align*}
\end{theorem}

\begin{proof}
    These are straightforward computations, which we provide in full in Appendix~\ref{app:gradient_descent}.
\end{proof}

\section{Conclusion}
\label{sec:conclusion}

To the best of our knowledge, 
this paper provides the first thorough investigation of monotonicity 
for a very general class of comparison-based preference learning,
with a focus on the effect of comparisons on the local minima,
and through the multiple facets of monotonicity.
While many previous papers pointed out deficiencies,
we highlighted a noteworthy desirable property of many models,
namely \emph{local pairwise monotonicity}.
We also provided insights into other forms of monotonicity.

\paragraph*{Limitations.}
While better improving the understanding of (non) monotonicity in preference learning, our theory does not capture other non-intuitive aspects, such as the changes of scores as shown in Figure~\ref{fig:experiment-chosen-rejected}.
Above all, we hope to motivate more work on the mathematical guarantees 
of preference learning algorithms,
in order to construct more trustworthy AIs~\cite{hoang2021tournesol}.
Also, we caution readers against the use of preference learning algorithms
from data collected in inhumane conditions, as is unfortunately mostly the case today~\cite{DW2025, perrigo2023exclusive, hao2023cleaning, ghanalawsuit}. 
The very existence of data annotators in their current working conditions is one of the most pressing social issues of AI training today, it is unclear whether our work could positively contribute to this issue.

\bibliographystyle{plain}
\bibliography{references.bib}

\newpage

\newpage
\appendix
\begin{center}
    {\Large \bf Supplemental material}
\end{center}
\section{Proofs of pairwise monotonicity for unequivocal comparisons}
\label{app:unequivocal}

\begin{proof}[Proof of Theorem~\ref{th:monotonicity_unequivocal}]
  Denote $\bD^\varepsilon \triangleq \bD \cup \varepsilon \set{(x, y, z, \max C)}$.
  We invoke the implicit function theorem 
  for the map $\Phi :\setR^{D+1} \to \setR^D, (\varepsilon, \theta) \mapsto \nabla_\theta \Loss (\theta | \bD^\varepsilon)$.
  Since $\nabla_{\theta} \Loss(\theta^* | \bD) = 0$, 
  we know that $\Phi(0, \theta^*) = 0$.
  The Jacobian matrix of $\Phi$ relative to $\theta$ is given by
  \begin{align}
    J_{\theta} \Phi (\varepsilon, \theta) = \nabla^2 \Loss (\theta | \bD^\varepsilon ).
  \end{align}
  
  Since we assumed $\nabla^2 \Loss (\theta | \bD )$ to be definite-positive, 
  we know it to be invertible.
  The implicit functions theorem thus applies, and provides the existence of $\varepsilon_0 > 0$ and a unique function $g: (-\varepsilon_0, \varepsilon_0) \to \setR^D$
  such that $g(0) = \theta^*$ and $\Phi(\varepsilon, g(\varepsilon)) = 0$ for all $\varepsilon \in (-\varepsilon_0, \varepsilon_0)$.
  Moreover, $g$ is differentiable and 
  \begin{align}
    g'(0) 
    &= - \left[ \partial_{\varepsilon} J_{\theta} \Phi(0, \theta^*) \right]^{-1} 
      \partial_\varepsilon \Phi (0, \theta^*) 
      = - \left[ \nabla^2 \Loss (\theta^* | \bD) \right]^{-1} 
      \partial_\varepsilon \nabla \Loss (\theta^* | \bD^\varepsilon)_{|\varepsilon = 0}
  \end{align}
  Now consider any $(x, y, z) \in \cB \times \cA \times \cA$,
  and define $\bD^\varepsilon \triangleq $
  \begin{align}
    \Loss (\theta | \bD^\varepsilon )
    &= \Loss (\theta | \bD) + \varepsilon \ell (s_{yz|x} (\theta), \max \cC).
  \end{align}
  It implies
  \begin{align}
    \nabla_\theta \Loss (\theta | \bD^\varepsilon )
    &= \nabla_\theta \Loss (\theta | \bD) 
      + \varepsilon \partial_s \ell(s_{yz|x} (\theta), \max \cC) \nabla_\theta s_{yz|x} (\theta).
  \end{align}
  Thus 
  \begin{align}
    \partial_\varepsilon \nabla_{\theta} \Loss (\theta | \bD^\varepsilon)_{|\varepsilon = 0}
    &= \partial_s \ell(s_{yz|x} (\theta), \max \cC) \nabla_\theta s_{yz|x} (\theta).
  \end{align}
  But by Assumption~\ref{ass:max_comparison}, 
  we know that $\partial_s \ell(s_{yz|x} (\theta), \max \cC) < 0$.
  In particular, we then have
  \begin{align}
    g'(0) 
    &= \alpha \left[ \nabla^2 \Loss (\theta^* | \bD) \right]^{-1} 
      \nabla_\theta s_{yz|x} (\theta^*),
  \end{align}
  where $\alpha = - \partial_s \ell(s_{yz|x} (\theta^*), \max \cC) > 0$.
  In particular, this implies that
  \begin{align}
    s_{yz|x} (\theta^\varepsilon) - s_{yz|x} (\theta^*)
    &= s_{yz|x} (g(\varepsilon)) - s_{yz|x} (g(0)) \\
    &= s_{yz|x} (g(0) + \varepsilon g'(0) + o(\varepsilon)) - s_{yz|x} (g(0)) \\
    &= \nabla_\theta s_{yz|x} (\theta^*)^T g'(0) \varepsilon + o(\varepsilon) \\
    &= \varepsilon \alpha \nabla_\theta s_{yz|x} (\theta^*)^T \left[ \nabla^2 \Loss (\theta^* | \bD) \right]^{-1} 
      \nabla_\theta s_{yz|x} (\theta^*) + o(\varepsilon),
      \label{eq:score_difference_first_order_approximation}
  \end{align}
  where we used the assumption that $s_{yz|x}$ was a differentiable function of $\theta$.
  We use again the fact that the Hessian matrix is definite positive, 
  along with the assumption that $\nabla_\theta s_{yz|x} (\theta^*) \neq 0$,
  which implies
  \begin{align}
    \beta \triangleq \alpha \nabla_\theta s_{yz|x} (\theta^*)^T \left[ \nabla^2 \Loss (\theta^* | \bD) \right]^{-1} 
    \nabla_\theta s_{yz|x} (\theta^*) > 0.
  \end{align}
  At last, we obtain $s_{yz|x} (\theta^\varepsilon) - s_{yz|x} (\theta^*) = \beta \varepsilon + o(\varepsilon)$
  with $\beta > 0$.
  Thus locally, up to redefining $\varepsilon_0$,
  we know that the score difference between $z$ and $y$ given $x$ strictly increases,
  as we add a small comparison intensification in favor of $z$.
\end{proof}

\section{Proofs of pairwise monotonicity for unequivocal comparisons}
\label{app:intensification}

\begin{proof}[Proof of Theorem~\ref{th:monotonicity_comparison_intensification}]
    The proof is very similar to the proof of Theorem~\ref{th:monotonicity_unequivocal},
    by now defining $\bD^\varepsilon \triangleq \bD + \Delta^{\varepsilon}_{y z|x}$
    We invoke the implicit function theorem 
    for the map $f : (\varepsilon, \theta) \mapsto \nabla_\theta \Loss (\theta | \bD^{\varepsilon} )$,
    which is a function $\setR^{1 + D} \to \setR^D$.
    Since $\nabla \Loss(\theta^*, \bD) = 0$, 
    we know that $f(0, \theta^*) = 0$.
    Note that its Jacobian matrix restricted to $\theta$ is given by
    \begin{align}
        J_{|\theta} (\varepsilon, \theta) 
        = \left[ \partial_{\theta_j} \partial_{\theta_i} \Loss (\theta | \bD^{\varepsilon} ) \right]_{i,j \in [D]},
    \end{align}
    which is exactly the Hessian matrix $\nabla^2 \Loss (\theta | \bD^{\varepsilon} )$.
    Since we assumed it to be positive-definite, 
    we know it to be invertible.
    Hence there exists $\varepsilon_0 > 0$ and a unique function $g: (-\varepsilon_0, \varepsilon_0) \to \setR^D$
    such that $g(0) = \theta^*$ and $f(\varepsilon, g(\varepsilon)) = 0$ for all $\varepsilon \in (-\varepsilon_0, \varepsilon_0)$.
    Moreover, $g$ is differentiable and 
    \begin{align}
        g'(0) 
        &= - \left[ \partial_{\varepsilon} J_{|\theta}(0, \theta^*) \right]^{-1} 
            \partial_\varepsilon f (0, \theta^*) 
        = - \left[ \nabla^2 \Loss (\theta^* | \bD) \right]^{-1} 
            \partial_\varepsilon \nabla \Loss (\theta^* | \bD^\varepsilon)_{|\varepsilon = 0}
    \end{align}
    Now assume also that $(x, y, z)$ appears exactly once in $\bD$.
    This can be done without loss of generality.
    Indeed, if it never appears, then the loss is unperturbed.
    If it appears multiple times, it suffices to add all the variations due to each appearance.
    Now, given $(x, y, z)$ appearing once in $\bD$, we have
    \begin{align}
        \Loss (\theta | \bD^{\varepsilon} )
        &= \Loss (\theta | \bD) + \left(
            \ell (s_{yz|x} (\theta), c + \varepsilon)
            - \ell (s_{yz|x} (\theta), c)
        \right).
    \end{align}
    It implies
    \begin{align}
        \nabla_\theta \Loss (\theta | \bD^{\varepsilon} )
        &= \nabla_\theta \Loss (\theta | \bD) + \left( 
            \partial_s \ell(s_{yz|x} (\theta), c + \varepsilon) 
            - \partial_s \ell(s_{yz|x} (\theta), c) 
        \right) \nabla_\theta s_{yz|x} (\theta).
    \end{align}
    Thus 
    \begin{align}
        \partial_\varepsilon \nabla_\theta \Loss (\theta | \bD^{\varepsilon})_{|\varepsilon = 0}
        &= \partial_c \partial_s \ell(s_{yz|x} (\theta), c) \nabla_\theta s_{yz|x} (\theta).
    \end{align}
    But by Assumption~\ref{ass:differentiable}, 
    we know that $\partial_c \partial_s \ell(s_{yz|x} (\theta), c) < 0$.
    In particular, we then have
    \begin{align} \label{eq:gprime0}
        g'(0) 
        &= \alpha \left[ \nabla^2 \Loss (\theta^* | \bD) \right]^{-1} 
            \nabla_\theta s_{yz|x} (\theta^*),
    \end{align}
    where $\alpha = - \partial_c \partial_s \ell(s_{yz|x} (\theta^*), c) > 0$
    In particular, this implies that
    \begin{align}
        s_{yz|x} (\theta^\varepsilon) - s_{yz|x} (\theta^*)
        &= s_{yz|x} (g(\varepsilon)) - s_{yz|x} (g(0)) 
        = s_{yz|x} (g(0) + \varepsilon g'(0) + o(\varepsilon)) - s_{yz|x} (g(0)) \\
        &= \nabla_\theta s_{yz|x} (\theta^*)^T g'(0) \varepsilon + o(\varepsilon) \\
        &= \varepsilon \alpha \nabla_\theta s_{yz|x} (\theta^*)^T \left[ \nabla^2 \Loss (\theta^* | \bD) \right]^{-1} 
            \nabla_\theta s_{yz|x} (\theta^*) + o(\varepsilon),
    \end{align}
    where we used the assumption that $s_{yz|x}$ was a differentiable function of $\theta$.
    We use again the fact that the Hessian matrix is definite positive, which implies
    \begin{align}
        \beta \triangleq \alpha \nabla_\theta s_{yz|x} (\theta^*)^T \left[ \nabla^2 \Loss (\theta^* | \bD) \right]^{-1} 
            \nabla_\theta s_{yz|x} (\theta^*) > 0.
    \end{align}
    At last, we obtain $s_{yz|x} (\theta^\varepsilon) - s_{yz|x} (\theta^*) = \beta \varepsilon + o(\varepsilon)$
    with $\beta > 0$.
    Thus locally, up to redefining $\varepsilon_0$,
    we know that the score difference between $z$ and $y$ given $x$ strictly increases,
    as we add a small comparison intensification in favor of $z$.
\end{proof}

\section{Global pairwise monotonicity for convex loss}
\label{app:strongly_convex}

\begin{proof}[Proof of Theorem~\ref{th:global}]
    Make Assumptions~\ref{ass:s_c1}, \ref{ass:max_comparison} and \ref{ass:strongly_convex},
    and let us focus on the first part of Theorem~\ref{th:global}.
    The latter part can be derived similarly.

    By strong convexity of the loss (Assumption~\ref{ass:strongly_convex}), 
    not only is the minimum $\theta^* (\bD)$ unique for all datasets $\bD$,
    the Hessian matrix $\nabla^2 \Loss(\theta^* (\bD) | \bD)$ 
    is also guaranteed to be definite positive.
    
    Now suppose that $\bD'$ is obtained from $\bD$ by $N$ operations,
    which are all either an addition of an unequivocal comparison to
    or a comparison intensification favors $y$ against $z$ under $x$.
    Denote $\bD_n$ the state of $\bD$ after the first $n$ operations.
    We define $f : [0, 1] \to \cD$ as follows.
    For $n \in \set{0, 1, \ldots, N-1}$ and $t \in [0, 1/N)$,
    we define $f(n/N + t) \triangleq \bD_n \cup (tN) \set{(x, y, z, \max \cC)}$.

    By Theorem~\ref{th:monotonicity_unequivocal},
    we know that $s_{yz|x} (\theta^*(f(t)))$ is locally nondecreasing for all $t \in [0, 1]$.
    More precisely, from its proof and especially~\eqref{eq:score_difference_first_order_approximation}, 
    we derive the fact that
    $s_{yz|x} (\theta^*(f(t)))$ is differentiable for all $t \in [0, 1]$ and that
    $\frac{d}{dt} s_{yz|x} (\theta^*(f(t))) \geq 0$ 
    (even if $\nabla_\theta s_{yz|x} (\theta^* (f(t))) = 0$).
    It follows that 
    \begin{align}
        0 &\leq \int_0^1 \frac{d}{dt} \left[ s_{yz|x} (\theta^*(f(t))) \right] dt \\
        &= s_{yz|x} (\theta^*(f(1))) - s_{yz|x} (\theta^*(f(0))) \\
        &= s_{yz|x} (\theta^*(\bD')) - s_{yz|x} (\theta^*(\bD)).
    \end{align}
    Rearranging the terms allows to conclude.
\end{proof}

\section{Proof of local individual score monotonicity}
\label{app:individual_score_monotonicity}

\begin{proof}[Proof of Theorem~\ref{th:individual_score_monotonicity}]
    The proof is very similar to the one of Theorem \ref{th:monotonicity_comparison_intensification}. Starting from \eqref{eq:gprime0}, we have then
        \begin{align}
        s_{z|x} (\theta^\varepsilon) - s_{z|x} (\theta^*)
        &= s_{z|x} (g(\varepsilon)) - s_{z|x} (g(0))  \\
        &= \nabla_\theta s_{z|x} (\theta^*)^T g'(0) \varepsilon + o(\varepsilon) \\
        &= \varepsilon \alpha \nabla_\theta s_{z|x} (\theta^*)^T \left[ \nabla^2 \Loss (\theta^* | \bD) \right]^{-1} 
            \nabla_\theta s_{zy|x} (\theta^*) + o(\varepsilon) \\
        &= \varepsilon \alpha e_z \nabla_\theta s_{|x} (\theta^*)^T \left[ \nabla^2 \Loss (\theta^* | \bD) \right]^{-1} 
            \nabla_\theta s_{|x} (\theta^*) e_{zy} + o(\varepsilon)
    \end{align}
    where the $e_z$ are elements of the canonical basis of $\mathbb{R}^D$.
    Finally, setting
    \begin{align}
        \beta \triangleq \alpha \nabla_\theta s_{z|x} (\theta^*)^T \left[ \nabla^2 \Loss (\theta^* | \bD) \right]^{-1} 
            \nabla_\theta s_{zy|x} (\theta^*)
    \end{align}
    and using the max-diagonal dominance of $ \nabla_\theta s_{|x} (\theta^*)^T \left[ \nabla^2 \Loss (\theta^* | \bD) \right]^{-1} 
            \nabla_\theta s_{|x} (\theta^*)$, we deduce that $\beta > 0$. This leads to $s_{z|x}(\theta^\epsilon) - s_{zy|x}(\theta^*) = \beta \epsilon + o(\epsilon)$ with $\beta > 0$. This allows to conclude similarly to Theorem \ref{th:monotonicity_comparison_intensification}.
\end{proof}

\section{Proof that fully-pairwise monotonicity implies individual-probability monotonicity}
\label{app:pairwisetoindividual}

\begin{proof}[Proof of Proposition~\ref{prop:pairwisetoindividual}]
Assuming probabilities are softmax functions of the scores, the implication follows from the fact that
\begin{align}
    \pi_\theta (y|x) 
    \triangleq \frac{\exp s_{y|x} (\theta)}{\sum_w \exp s_{w|x} (\theta)}
    = \frac{1}{1 + \sum_{w \neq y} \exp \left( - s_{yw|x} (\theta) \right)},
\end{align}
which is an increasing function of the $s_{yw|x}$'s, for $w \in \cA$.

Hence, $\pi_\theta(y|x)$ inherits the fully pairwise monotonicity of the scores and we have $\pi_{\theta^\epsilon}(y | x) \ge \pi_{\theta}(y | x)$. The proof for $z$ is similar.
\end{proof}

\section{Proof that GBT is fully-pairwise monotone}
\label{app:GBTfullypairwise}

The proof of Proposition \ref{prop:GBTfullypairwise} relies on the following result for diagonally dominant matrices.

\begin{lemma}
\label{lem:MNlemma}
   Let $M$ be a symmetric and strictly diagonally dominant matrix
   (i.e. $|M_{yy}| > \sum_{z \neq y} |M_{yz}|$ for any $y$)
   such that $M_{yy} > 0$ and $M_{yz} \leq 0$ for any $y \neq z$. 
   Then, its inverse $N$ satisfies
    \begin{align}\label{eq:propMfully}
        N_{yy} - N_{yz} \geq N_{wy} - N_{wz}
    \end{align}
    for any $y,z,w \in \mathcal{A}$.
\end{lemma}

\begin{proof}
        We first prove the following result. Assume that $a$ is a vector such that $\max_v a_v > 0$ and denote $w = \arg \max_v a_v$ so that $a_w > 0$. Then, the vector $b = M a$ is such that $b_w > 0$.
        Assume by contradiction that $b_w \leq 0$. Then, we have 
        \begin{align}
            M_{ww} a_w = - \sum_v M_{wv} a_v + b_w \leq  - \sum_v M_{wv} a_v.
        \end{align}
        However, we also have
        \begin{align}
            \sum_v (-M_{wv}) a_v \leq a_w \sum_v (-M_{wv}) < a_w M_{ww}
        \end{align}
        by strict diagonal dominance and using that $a_v \leq a_w$ for any $v$ and $-M_{wv}>0$. The two inequalities are contradictory, hence $b_w > 0$. \\

        We apply this result to $a = N_{y} - N_z$, the difference of the two columns $N_y$ and $N_z$ of $N$. The latter being the inverse of $M$, we have $M a = b = e_{yz}$ where the $e_y$ are the element of the canonical basis. First, we observe that $a_y = N_{yy} - N_{yz} > 0$ due to \cite[Lemma 1]{fageotGeneralizedBradleyTerryModels2024}. Since $y$ is the only index $w$ for which $b_w = 1 > 0$, we deduce from the previous result that $y = \arg \max_w a_w = \arg\max_w N_{wy} - N_{wz}$, which gives precisely \eqref{eq:propMfully}.
\end{proof}

\begin{proof}[Proof of Proposition \ref{prop:GBTfullypairwise}]
We can follow the proof of \cite[Theorem 2]{fageotGeneralizedBradleyTerryModels2024}
and use Lemma \ref{lem:MNlemma} instead of \cite[Lemma 1]{fageotGeneralizedBradleyTerryModels2024} to conclude.    
\end{proof}

\section{Gradient descent monotonicity}
\label{app:gradient_descent}

\begin{proof}[Proof of Theorem~\ref{th:gradient_descent}]
In this section, we assume $\cR = 0$,
and we consider the impact of sampling $(x, y, z, \max \cC)$
and of performing an infinitesimal stochastic gradient step 
with respect to this sample.
More specifically, consider any solution $\theta \in \setR^D$.
The infinitesimal stochastic gradient step then yields
\begin{align}
    \theta(t+dt)
    &= \theta(t) - \nabla_\theta \left[ \ell (s_{yz|x} (\theta), \max \cC) \right] dt,
\end{align}
which we can rewrite
\begin{align}
    \frac{d}{dt} \theta 
    &= - \nabla_\theta \left[ \ell (s_{yz|x} (\theta), \max \cC) \right] 
    = \alpha \nabla s_{yz|x},
\end{align}
with $\alpha \triangleq - \partial_s \ell (s_{yz|x} (\theta), \max \cC) > 0$.
We then have
\begin{align}
    \frac{d}{dt} s_{yz|x}
    &= \nabla s_{yz|x} \cdot \frac{d\theta}{dt}
    = \alpha \norm{\nabla s_{yz|x} (\theta)}{2}^2, \\
    \frac{d}{dt} s_{y|x}
    &= \nabla s_{y|x} \cdot \frac{d\theta}{dt} 
    = \alpha \left( 
       \norm{\nabla s_{y|x} (\theta)}{2}^2 
       - \nabla s_{y|x} (\theta) \cdot \nabla s_{z|x} (\theta)
    \right), \\
    \frac{d}{dt} s_{yw|x}
    &= \nabla s_{yw|x} \cdot \frac{d\theta}{dt} 
    = \alpha \left( 
       \nabla s_{yw|x} (\theta) \cdot \nabla s_{yz|x} (\theta)
    \right).
\end{align}
Respectively, if the right-hand sides are strictly positive, 
then gradient-descent monotonicity holds, 
respectively for pairwise, individual-score and 
fully-pairwise monotonicity.
\end{proof}

\end{document}